\documentclass[reqno,centertags,12pt]{amsart}
\usepackage{amssymb,amsmath,amsthm,dsfont}
\makeatletter \def\@strippedMR{} \def\@scanforMR#1#2#3\endscan{%
  \ifx#1M\ifx#2R\def\@strippedMR{#3}%
  \else\def\@strippedMR{#1#2#3}%
  \fi\fi} \renewcommand\MR[1]{\relax \ifhmode\unskip\spacefactor3000
  \space\fi \@scanforMR#1\endscan
  MR\MRhref{\@strippedMR}{\@strippedMR}} \makeatother

\newcommand{\R}{\mathbb{R}} \newcommand{\Z}{\mathbb{Z}}
\newcommand{\T}{\mathbb{T}} \newcommand{\C}{\mathbb{C}}

\theoremstyle{plain} \newtheorem{theorem}{Theorem}[section]
\newtheorem{lemma}[theorem]{Lemma}
\newtheorem{coro}[theorem]{Corollary}
\newtheorem{prop}[theorem]{Proposition}

\theoremstyle{definition} \newtheorem{definition}[theorem]{Definition}
\theoremstyle{remark} 

\DeclareMathOperator{\supp}{supp}

\newcommand{\ls}{\lesssim}

\begin{document}

\title[Strichartz estimates in $4d$ and applications]{Strichartz
  estimates for partially periodic solutions to Schr\"odinger
  equations in $4d$ and applications}
\author[S.~Herr]{Sebastian~Herr}\author[D.~Tataru]{Daniel~Tataru}
\author[N.~Tzvetkov]{Nikolay~Tzvetkov}

\subjclass[2000]{35Q55}

\address{Mathematisches Institut, Universit\"at Bonn, Endenicher Allee
  60, 53115 Bonn, Germany} \email{herr@math.uni-bonn.de}
\address{Department of Mathematics, University of California,
  Berkeley, CA 94720-3840, USA} \email{tataru@math.berkeley.edu}
\address{D\'epartement de Math\'ematiques, Universit\'e de
  Cergy-Pontoise, 2, avenue Adolphe Chauvin, 95302 Cergy-Pontoise
  Cedex, France} \email{nikolay.tzvetkov@u-cergy.fr}

\begin{abstract} 
  We consider the energy critical nonlinear Schr\"odinger equation on
  periodic domains of the form $\R^m\times \T^{4-m}$ with $m =
  0,1,2,3$.  Assuming that a certain $L^4$ Strichartz estimate holds
  for solutions to the corresponding linear Schr\"odinger equation, we
  prove that the nonlinear problem is locally well-posed in the energy
  space.  Then we verify that the $L^4$ estimate holds if $m=2,3$,
  leaving open the cases $m = 0,1$.
 \end{abstract}
\keywords{Cubic Nonlinear Schr\"odinger Equation, global well-posedness, Strichartz estimate}
\maketitle
\section{Introduction and main results}\label{sect:intro-main}

The nonlinear Schr\"odinger equation (NLS) on a Riemannian
manifold $(M,g)$ has the form
\begin{equation}\label{eq:nlsgen}
  (i\partial_t+\Delta_g)u=\pm |u|^{p-1} u,\quad  u(0,x)=u_0(x),
\end{equation}
where the $\pm$ signs correspond to the defocusing, respectively the
focusing case. This admits a conserved energy functional, namely
\[
E(u) = \int_M \frac12 |\nabla u|^2 \pm \frac{1}{p+1} |u|^{p+1} dV
\]
which is also the Hamiltonian for this problem.

The problem \eqref{eq:nlsgen} is called energy critical when the
energy is invariant with respect the natural associated scaling.  This
corresponds to the exponent
\[
p = \frac{d+2}{d-2}, \qquad d \geq 3
\]
where $d$ is the dimension of $M$.  For energy critical problems it is
natural to study the local and global well-posedness for initial data
in the energy space $\dot H^1(M)$.

In $\R^d$, the fact that the energy critical NLS is locally well-posed
in the energy space is a straightforward consequence of the Strichartz
estimates for the linear Schr\"odinger equation, which are a
quantitative expression of the dispersive character of the
equation, see e.g. \cite{C03,KT98}. In the simplest form, the Strichartz estimates for $M= \R^d$
can be expressed in the form
\begin{equation}\label{st}
  \| P_{C} e^{it\Delta} u_0\|_{L^q (\R \times M)} \lesssim 
  \lambda^{(d+2)(\frac{1}{q_d} - \frac{1}q)}\|  P_{C} u_0\|_{L^2}, \quad
  q \geq q_d, \, C \in \mathcal R(\lambda),
\end{equation}
where $P_{C}$ stands for a (spatial) frequency projector to the cube $C$,
$\mathcal R(\lambda)$ is the family of all cubes of size $\lambda$,
and $q_d$ is the Strichartz exponent
\[
q_d = \frac{2(d+2)}d
\]
This is invariant with respect to frequency translations, which is a
consequence of the Galilean invariance.  We refer to the extreme case
$q = q_d$ as the sharp Strichartz estimate; all others follow from it
by Bernstein's inequality.

By contrast, in compact domains there is less dispersion as the linear
waves cannot spread toward spatial infinity. In particular, the sharp
Strichartz estimates are unlikely to hold; weaker forms of the
estimates may still hold locally in time, but seem to depend on
delicate spectral and geometric properties of the manifold. We refer
the reader to \cite{B93a,BGT04} for positive results in this
direction, and also to \cite{B93a,TT01} for counterexamples to
\eqref{st} in the case $q=q_d$ on
$M=\T^d$, $d=1,2$.

In our previous work \cite{HTT10a} we have considered the quintic
$H^1$ critical NLS on the three dimensional (rational) torus $\T^3$,
and proved that this problem is locally well-posed for initial data in
the energy space, and globally well-posed for small data. This was the first critical well-posedness result
for NLS evolving in a compact spatial domain.  Our goal in this
article is to extend the result of \cite{HTT10a} to the cubic energy
critical NLS in four dimensions, namely 
\begin{equation}\label{eq:nls}
  (i\partial_t+\Delta_g)u=\pm |u|^{2} u,\quad  u(0,x)=u_0(x),
\end{equation}
 for manifolds $M$ of the form $M =
\R^m\times\T^n$, $m+n=4$ (with the convention that in the case $m=0$,
$M=\T^4$ while for $m=4$, $M=\R^4$).

What helps in the case of the energy critical problem is that the the
sharp Strichartz estimates are not necessary; instead it suffices to
have some weaker nonsharp bounds, provided that they are still
consistent with scaling. Our starting point in this direction is the
work of Bourgain~\cite{B93a}, who, in the case of periodic domains $\T^n$,
proved some scale invariant Strichartz estimates. Precisely, Bourgain
showed that the bound \eqref{st} holds for $q > 4$ on $\T^3$ and for
$q \geq 4$ on $\T^4$. One key to both our earlier result in
\cite{HTT10a} and our present result is an extension of \eqref{st},
which is explained next.

For $1 \leq \mu \leq \lambda$ let $\mathcal{R}(\mu,\lambda)$ be the
collection of all rectangles in $\R^m\times\Z^n$ of which have one
side equal to $\mu$ and all other $d-1$ sides equal to
$\lambda$. More precisely, we define $R\in \mathcal{R}(\mu,\lambda)$ as
\begin{equation}\label{rlmu}
R= \{\xi \in \R^m\times\Z^n:\xi-\xi_0\in [-\lambda,\lambda]^d,
\, |a\cdot \xi
-c|\leq \mu \}
\end{equation}
for some $\xi_0,a\in \R^d$, $|a|=1$, $c\in \R$.  With these notations
we consider the following extension of the bound \eqref{st}:
\begin{equation}\label{st1}
  \| P_{R} e^{it\Delta} u_0\|_{L^q (\R \times M)} \lesssim 
  \lambda^{(d+2)(\frac{1}{q_d} - \frac{1}q)} \left(\frac{\mu}{\lambda}\right)^\delta
  \|  P_{R} u_0\|_{L^2}, \quad
  R \in \mathcal{R}(\mu,\lambda)
\end{equation}
Here $\delta > 0$ is a small constant.

In the three dimensional result in \cite{HTT10a} it was sufficient to
prove that \eqref{st1} with $p = 4$ holds for $q > 4$. In four
dimensions we have to deal with the more difficult case $p=4$, which
corresponds to the endpoint of Bourgain's result. Our first result is
a conditional result:

\begin{theorem}\label{thm1}
  Let $M=\R^m\times\T^n$, $0\leq m\leq 4$, $m+n=4$. Assume that the
  bound \eqref{st1} with $q = 4$ holds uniformly for all $1 \leq \mu
  \leq \lambda$ with some $\delta > 0$.  Then the Cauchy problem
  \eqref{eq:nls} is globally well-posed for small initial data $u_0\in
  H^s(M)$, $s\geq 1$.
\end{theorem}

This result includes the existence of mild solutions which are unique
in a certain space, which form a continuous curve in $\dot H^1(M)$,
and depend Lipschitz-continuously on the initial data. We also have a
suitable large data result with a life span depending on the profile
of the initial datum and not only on its $\dot H^1(M)$ norm, see
e.g. \cite{HTT10a} for a precise statement.
 
The global in time nature of our result results from the energy
conservation which controls the $\dot H^1(M)$ norm, a quantity which is
critical with respect to the natural scaling $u(t,x)\rightarrow
\lambda u(\lambda^2 t, \lambda x)$ of solutions to the nonlinear
equation \eqref{eq:nlsgen}.
 
We refer to \cite{B93c,BGT04,GP} for well-posedness results of some four
dimensional NLS on compact spatial domains, in spaces of sub-critical
regularity.

We remark that even using the bound \eqref{st1} with $q = 4$, the
proof of Theorem~\ref{thm1} is not a foregone conclusion.  Instead it
requires a delicate functional space set-up based on $U^2$ and $V^2$
type spaces. In addition the almost orthogonality arguments with
respect to both spatial and time frequencies, used in the proof of
Proposition~\ref{thm:bil-str-ext} below, are crucial to make work our
approach in four dimensions (recall that in the sub-critical analysis
of \cite{B93a,B93c} one only uses almost orthogonality with respect to
the spatial frequency).

Given the above theorem, the interesting question becomes to verify
whether the bound \eqref{st1} holds with $q = 4$, which for convenience
we restate as 
\begin{equation}\label{strichartz}
  \| P_{R} e^{it\Delta} u_0\|_{L^4 (\R \times M)} \lesssim 
  \lambda^{\frac12} \left(\frac{\mu}{\lambda}\right)^\delta
  \|  P_{R} u_0\|_{L^2}, \quad
  R \in \mathcal{R}(\mu,\lambda)
\end{equation}
In this paper we will only verify \eqref{strichartz} in
some partially periodic situations. Here is the statement.
\begin{theorem}\label{thm2}
Let $M=\R^m\times\T^n$ with  $m+n=4$ and $n \leq 2$. 
Then  the
  bound \eqref{strichartz}  holds uniformly for all $1 \leq \mu
  \leq \lambda$ with some $\delta > 0$.
\end{theorem}
Let us remark that Theorems~\ref{thm1} and \ref{thm2} remain valid in
the case of $m$-dimensional irrational tori (replacing $\T^m$).

Unfortunately we are not able at this time to deal with the $4$
dimensional torus $\T^4$, where $\T=\R/2\pi \Z$.  The analysis of the
models with $n=1,2$ is easier than the case of $\T^4$ but they still
represent spatial domains of strongly non euclidean global geometry,
compared for instance with the asymptotically euclidean geometry.

We remark that in the case $M=\R^4$ the estimate \eqref{strichartz}
with $\delta=1/12$ is a simple
consequence of the $L^3$ Strichartz estimate and Bernstein's
inequality (in fact it can be shown to hold true with $\delta=\frac14$
by the strategy of proof presented in Section \ref{sect:mult-str}).  In the case $M=\T^4$ estimate \eqref{strichartz} with
$\delta=0$ is contained in the work of Bourgain \cite{B93a}.  There
are two possible approaches to obtain the improvement with
$\delta>0$. The first is to refine Bourgain's reasoning at $L^4$
level.  The second is to apply an argument similar to \cite{HTT10a}.
This would consist in first proving the scale invariant inequality 
\eqref{st} for some $p<4$ (such an inequality is
conjectured at page 119 of Bourgain's paper \cite{B93a}) and then
interpolate with a suitable inequality for $p$ a large even integer, a
situation where refinements sensible to finer scales are much easier
to be established.

We conclude this introduction by introducing some notations.

If $M=\R^{m}\times \T^{n}$ then we denote
$\widehat{M}:=\R^{m}\times\Z^{n}$.
For $f:M\to \C$ an integrable function, we define
the Fourier transform $\hat{f}$ of $f$ by
\[
\hat{f} (\xi):=(2\pi)^{-\frac{m+n}{2}}\int_{ \R^m \times [0,2\pi]^{n}}
e^{-ix\cdot \xi} f(x) \;dx,\; \xi \in \R^{m} \times\Z^{n}.
\]
Let $\psi\in C_0^\infty(-2,2)$ be
a non-negative, even function with $\psi(s)=1$ for $|s|\leq 1$. We use
$\psi$ to construct a partition of unity $(\psi_\lambda)$ on $L^2(M)$
as follows: For a dyadic number $\lambda\geq 1$ we define
\[
\psi_\lambda(\xi)=\psi(|\xi|/\lambda)-\psi(2|\xi|/\lambda), \quad
\text{for } \lambda\geq 2, \quad \psi_1(\xi)=\psi(|\xi|).
\]
We define the frequency localization operators $P_\lambda:L^2(M)\to
L^2(M)$ as the Fourier multiplier with symbol $\psi_\lambda$, and for
brevity we also write $u_\lambda:=P_\lambda u$. Moreover, we define
$P_{\leq \lambda}:=\sum_{1\leq \mu\leq \lambda}P_\mu$ (dyadic sum).
More generally, for any measurable set $S\subset \widehat{M}$ we
define the Fourier projection operator $P_S$ with symbol $\chi_S$,
where $\chi_S$ denotes the (sharp) characteristic function of $S$. Let
$s \in \R$.  We define the Sobolev space $H^s(M)$ as the space of all
$L^2(M)$-functions for which the norm
\[
\|f\|_{H^s(M)}:= \left(\sum_{\lambda \geq 1}\lambda^{2s} \|P_\lambda
  f\|_{L^2(M)}^2\right)^{\frac12}
\]
is finite (the summation runs over dyadic values of $\lambda$).
Throughout this paper we will mainly use the greek letters $\lambda$
and $\mu$ to denote dyadic numbers.
\section{Critical function spaces and proof of
  Theorem~\ref{thm1}}\label{sect:appl}
\noindent
Let $\mathcal{Z}$ be the set of finite partitions
$-\infty<t_0<t_1<\ldots<t_K\leq \infty$ of the real line. Let
$\chi_I:\R\to\R$ denote the sharp characteristic function of a set
$I\subset \R$. The following definitions are as in \cite[Section
2]{HHK09}, and \cite[Section 2]{HTT10a}.
\begin{definition}\label{def:u}
  Let $1\leq p <\infty$, and $H$ be a complex Hilbert space.  A
  $U^p$-atom is a piecewise defined function $a:\R \to H$,
  \begin{equation*}
    a=\sum_{k=1}^K\chi_{[t_{k-1},t_k)}\phi_{k-1}
  \end{equation*}
  where $\{t_k\}_{k=0}^K \in \mathcal{Z}$ and $\{\phi_k\}_{k=0}^{K-1}
  \subset H$ with $\sum_{k=0}^{K-1}\|\phi_k\|_{H}^p=1$.
 
  The atomic space $U^p(\R,H)$ consists of all functions $u: \R \to H$
  such that
  \begin{equation*}
    u=\sum_{j=1}^\infty \lambda_j a_j \; \text{ for } U^p\text{-atoms } a_j,\;
    \{\lambda_j\}\in \ell^1,
  \end{equation*}
  with norm
  \begin{equation*}
    \|u\|_{U^p}:=\inf \left\{\sum_{j=1}^\infty |\lambda_j|
      :\; u=\sum_{j=1}^\infty \lambda_j a_j,
      \,\lambda_j\in \C,\; a_j \text{ $U^p$-atom}\right\}.
  \end{equation*}
\end{definition}
\begin{definition}\label{def:v}
  Let $1\leq p<\infty$, and $H$ be a complex Hilbert space.
  \begin{enumerate}
  \item We define $V^p(\R,H)$ as the space of all functions $v:\R\to
    H$ such that
    \begin{equation*}
      \|v\|_{V^p}
      :=\sup_{\{t_k\}_{k=0}^K \in \mathcal{Z}} \left(\sum_{k=1}^{K}
        \|v(t_{k})-v(t_{k-1})\|_{H}^p\right)^{\frac{1}{p}}<+\infty,
    \end{equation*}
    where we use the convention $v(\infty)=0$.
  \item We denote the closed subspace of all right-continuous
    functions $v:\R\to H$ such that $\lim_{t\to -\infty}v(t)=0$ by
    $V^p_{rc}(\R,H)$.
  \end{enumerate}
\end{definition}
\begin{definition}\label{def:delta_norm}
  For $s \in \R$ we let $U^p_\Delta H^s $ resp. $V^p_\Delta H^s$ be
  the spaces of all functions $u:\R\to H^s(M)$ such that $t \mapsto
  e^{-it \Delta}u(t)$ is in $U^p(\R,H^s)$ resp. $V^p_{rc}(\R,H^s)$,
  with norms
  \begin{equation*}
    \| u\|_{U^p_\Delta H^s} = \| e^{-it \Delta} u\|_{U^p(\R,H^s)},
    \qquad 
    \| u\|_{V^p_\Delta H^s} = \| e^{-it \Delta} u\|_{V^p(\R,H^s)}.
  \end{equation*}
\end{definition}
As an example, we show that the $L^4$-estimate \eqref{strichartz} for
linear solutions has a straightforward extension to
$U^4_\Delta$-functions:
\begin{coro}\label{coro:l4u4}
  Assume that the bound \eqref{strichartz} holds. Then we also have
  \begin{equation}\label{eq:l4u4}
    \|P_R u\|_{L^4([0,1]\times M)}\ls
    \lambda^{\frac12-\delta}\mu^{\delta}\|u\|_{U^4_\Delta}.
  \end{equation}
\end{coro}
\begin{proof}
  Due to the atomic structure of $U^4$ it suffices to prove
  \eqref{eq:l4u4} for atoms, i.e. for piecewise solutions of the
  linear equation, i.e.
  \[
  a(t)=\sum_{k=1}^K \chi_{[t_{k-1},t_k)}(t)e^{it\Delta} \phi_{k-1},
  \text{ with } \sum_{k=1}^K\|\phi_{k-1}\|^4_{L^2}=1.
  \]
  For $I_k=[0,1]\cap [t_{k-1},t_k)$ we have
  \begin{align*}
    \|P_R a\|^4_{L^4([0,1]\times M)}=& \sum_{k=1}^K \|P_R e^{it\Delta}\phi_{k-1}\|^4_{L^4(I_k \times M)}\\
    \ls &\sum_{k=1}^K
    \lambda^{2-4\delta}\mu^{4\delta}\|\phi_{k-1}\|_{L^2}^4\\
    \ls &\lambda^{2-4\delta}\mu^{4\delta},
  \end{align*}
  which proves the claim.
\end{proof}
Similarly to \cite{HTT10a}, we define modifications of $U^2_\Delta$
and $V^2_\Delta$ which are better adapted to the finer localizations
we need to consider.  For $z\in \Z^4$ we define the cube
$C_z=z+[0,1)^4$, which induces a disjoint partition $\cup_{z \in
  \Z^4}C_z =\R^4$.  For a function $u: \R \to H^s(M)$ we consider for
every $z \in \Z^4$ the map \[Q_z(u):\R\to H^s(M),\; Q_z(u)(t)=P_{C_z}
u(t).\]
\begin{definition}\label{def:xs_ys}
  Let $s\in \R$ be given.
  \begin{enumerate}
  \item We define $X^{s}$ as the space of all functions $u: \R \to
    H^s(M)$ such that $Q_z(u)\in U^2_\Delta(\R,H^s(M))$ for every $z
    \in \Z^d$, and
    \begin{equation*}
      \|u\|_{X^{s}}:=\left(\sum_{z \in \Z^d}
        \|Q_z(u)\|_{U^2_\Delta(\R,H^s)}^2\right)^{\frac12}<+\infty.
    \end{equation*}
  \item We define $Y^{s}$ as the space of all functions $u:\R \to
    H^s(M)$ such that $Q_z(u)\in V^2_{\Delta}(\R,H^s(M))$ for every
    $z\in \Z^d$, and
    \begin{equation*}
      \|u\|_{Y^{s}}
      :=\left(\sum_{z\in \Z^d}\|Q_z(u)\|_{V^2_\Delta(\R,H^s)}^2\right)^{\frac12}<+\infty.
    \end{equation*}
  \end{enumerate}
\end{definition}
For a time interval $I\subset \R$ we also consider the restriction
spaces $X^s(I)$ and $Y^s(I)$, defined in the natural manner.  The next
statement results from the definition.
\begin{prop}\label{prop:x_in_y}
  The following embeddings are continuous:
  \[
  U^2_\Delta H^s \hookrightarrow X^s \hookrightarrow Y^s
  \hookrightarrow V^2_\Delta H^s.
  \]
\end{prop}
The motivation for the introduction of the $X^s$ and $Y^s$ spaces lies
in the following.
\begin{coro}\label{coro:part_prop}
  Let $\{S_k\}$ be a partition of $\R^d$ into measurable sets $S_k$
  with the property \[\sup_{z \in \Z^d } \#\{k: C_z\cap S_k
  \not=\emptyset \}<+\infty.\] Then
  \begin{equation*}
    \sum_k \| P_{S_k} u\|_{V^2_\Delta H^s}^2 \lesssim  \| u\|^2_{Y^s}.
  \end{equation*}
\end{coro}
Now, we are able to state and prove the key estimate.
\begin{prop}\label{thm:bil-str-ext}
  Under the hypothesis of Theorem~\ref{thm1}, there exists $\delta>0$
  such that for all dyadic $\lambda\geq\mu \geq 1$ it holds
  \begin{equation*}
    \|P_\lambda u_1 P_\mu u_2\|_{L^2([0,1] \times M )}
    \ls \mu \Big(\frac{\mu}{\lambda}+\frac1\mu\Big)^\delta \|P_\lambda u_1\|_{Y^0}\|P_\mu u_2\|_{Y^0}.
  \end{equation*}
\end{prop}
\begin{proof}[Proof of Proposition~\ref{thm:bil-str-ext}]
  Set $I=[0,1]$.  By almost orthogonality and Corollary
  \ref{coro:part_prop} we may restrict $P_\lambda u_1$ to a cube $C
  \in \mathcal{R}(\mu)$ of side-length $\mu$, and reduce the claim to
  the estimate
  \begin{equation*}
    \|P_\lambda P_C u_1 P_\mu u_2\|_{L^2(I\times M)}\ls \mu(\mu/\lambda+1/\mu)^{\delta}\|P_\lambda P_C u_1\|_{V^2_\Delta}\|P_\mu u_2\|_{V^2_\Delta},
  \end{equation*}
  which follows by interpolation \cite[Proposition 2.20]{HHK09} from
  the two estimates
  \begin{equation}\label{eq:red_u2}
    \|P_\lambda P_C u_1 P_\mu u_2\|_{L^2(I\times M)}\ls  \mu(\mu/\lambda+1/\mu)^{\delta}\|u_1\|_{U^2_\Delta}\|u_2\|_{U^2_\Delta},
  \end{equation}
  and
  \begin{equation}\label{eq:red_u4}
    \|P_\lambda P_C u_1 P_\mu u_2\|_{L^2(I\times M)}\ls \mu \|u_1\|_{U^4_\Delta}\|u_2\|_{U^4_\Delta},
  \end{equation}
  which we will prove next. Indeed, estimate \eqref{eq:red_u4} is
  immediately obtained from the $L^4$-bound \eqref{eq:l4u4} and the
  Cauchy-Schwarz inequality, because
  \begin{align*}
    \|P_\lambda P_C u_1 P_\mu u_2\|_{L^2(I\times M)}\leq & \|P_\lambda P_C u_1\|_{L^4(I\times M)}\| P_\mu u_2\|_{L^4(I\times M)}\\
    \lesssim & \mu^\frac12\|u_1\|_{U^4_\Delta}
    \mu^\frac12\|u_2\|_{U^4_\Delta},
  \end{align*}
  where we have used the fact that both $P_\lambda P_C u_1$ and $P_\mu
  u_2$ have frequency support in cubes of side-lengths proportional to
  $\mu$.

  It remains to prove the estimate \eqref{eq:red_u2}. By
  \eqref{eq:red_u4} it suffices to consider the case $\mu \ll
  \lambda$. Due to the atomic structure of $U^2$ (see also
  \cite[Proposition 2.19]{HHK09} for more details) it is enough to
  prove the corresponding estimate for solutions to the linear
  equation, i.e
  \begin{align*}
    \|P_C e^{it\Delta}\phi_1 e^{it\Delta}\phi_2\|_{L^2(I\times M)}\ls
    \mu(\mu/\lambda+1/\mu)^{\delta}\|\phi_1\|_{L^2}\|\phi_2\|_{L^2}.
  \end{align*}
  with initial data satisfying
  \[\supp \widehat{\phi}_1\subset \{|\xi|\approx \lambda\},
  \quad \supp \widehat{\phi}_2\subset \{|\xi|\approx \mu\}.\] We
  extend the functions to the real line. For this we consider a 
 Schwartz function $\psi$ which is frequency localized in $[-1,1]$
and which is nonzero on $I$, and we define
 \[
  u_j(t)=\psi(t)e^{it\Delta}\phi_j
\]
Then we will  prove instead the stronger bound
  \begin{align*}
    \| P_C u_1 u_2\|_{L^2(\R\times M)}\ls
    \mu(\mu/\lambda+1/\mu)^{\delta}\|\phi_1\|_{L^2}\|\phi_2\|_{L^2}.
  \end{align*}

  Let $C=(\xi_0+[-\mu,\mu]^4)\cap \widehat{M}$ for some $\xi_0\in
  \R^4$, $|\xi_0|\approx \lambda$. We decompose $C$ into almost
  disjoint strips $R_k$ of width $\nu=\max\{\mu^2/\lambda,1\}$, which
  are orthogonal to $\xi_0$, i.e.
  \[
  C=\bigcup_{k\in \Z: |k| \approx \lambda/\nu} R_k, \quad R_k=\{\xi\in
  C: |\xi\cdot a -\nu k-|\xi_0||\leq \nu \}, \, a:=\xi_0|\xi_0|^{-1}.
  \]
  Then $R_k\in \mathcal{R}(\nu,\mu)$, and we have
  \[
  P_Cu_{1}u_{2}=\sum_{k\in \Z: |k| \approx
    \lambda/\nu} P_{R_k}u_{1}u_{2}.
  \]
  We observe that there is $L^2(\R\times M)$ almost orthogonality in
  this decomposition. Indeed, for $ (\tau_1,\xi_1)  \in \supp \widehat{P_{R_k}u_{1}}$
 we have $\xi_1 \in  R_k$  and $|\tau_1+\xi_1^2| \leq 1$. Furthermore, 
we compute 
  \[
  |\xi_1|^2=\underbrace{(\xi_1\cdot a)^2}_{=\nu^2 k^2+O(\nu^2 k)
  }+\underbrace{|\xi_1-\xi_0|^2}_{=O(\mu^2)}-\underbrace{((\xi_1-\xi_0)\cdot
    a)^2}_{=O(\mu^2)}.
  \]
  In addition, we recall that $\mu^2\ls \nu^2|k|$, which implies
  \[
  |\tau_1+\nu^2k^2|=|\tau_1+|\xi_1|^2|+O(\nu^2 |k|)=O(\nu^2|k|).
  \]
  Therefore, the time frequency  of $P_{R_k}u_{1}$ is
  localized to the interval
  $[-\nu^2k^2-c\nu^2|k|,-\nu^2k^2+c\nu^2|k|]$. The second factor
  $u_{2}$ has time frequency $\tau_2$ localized the interval 
   $|\tau_2| \ls \mu^2\ls \nu^2|k|$. Hence 
\[
\tau_1+ \tau_2 \in [-\nu^2k^2-c\nu^2|k|,-\nu^2k^2+c\nu^2|k|]
\]
These intervals are essentially disjoint, therefore the functions
  $\{P_{R_k}u_{1}u_{2}\}_{k}$ are almost orthogonal
  in $L^2(\R\times M)$, i.e.
  \[
  \|P_Cu_{1}u_{2}\|_{L^2}^2\approx \sum_{k\in \Z:
    |k| \approx \lambda/\nu}
  \|P_{R_k}u_{1}u_{2}\|_{L^2}^2.
\] 
By
  Cauchy-Schwarz and Corollary \ref{coro:l4u4} we obtain
  \begin{align*}
    \|P_{R_k}u_{1}u_{2}\|_{L^2}\ls&\  \|P_{R_k}u_{1}\|_{L^4}\|u_{2}\|_{L^4}\\
    \ls&\ \mu^{\frac12} \Big(\frac{\nu}{\mu}\Big)^{\delta}
    \|P_{R_k}u_{1}\|_{U_\Delta^4}\mu^{\frac12}\|u_{2}\|_{U_\Delta^4}\\
    \ls&\ \mu(\mu/\lambda+1/\mu)^{\delta}
    \|P_{R_k}\phi_{1}\|_{L^2}\|\phi_2\|_{L^2},
  \end{align*}
  where in the last inequality we used the fact that $\nu\leq
  \mu^2/\lambda+1$. This finishes the proof of
  Proposition~\ref{thm:bil-str-ext}.
\end{proof}
In order to apply our estimates to the integral formulation of the NLS problem
\eqref{eq:nls} we summarize some properties of our spaces.  The
following Proposition follows directly from the atomic structure of
$U^2$.
\begin{prop}\label{prop:linear}
  Let $s\geq 0$, $0<T\leq \infty$ and $\phi \in H^s(M)$. Then, for the
  linear solution $u(t):=e^{it\Delta}\phi$ for $t\geq 0$ we have $u
  \in X^s([0,T))$ and
  \begin{equation}\label{eq:linear}
    \|u\|_{X^s([0,T))}\leq \|\phi\|_{H^s}.
  \end{equation}
\end{prop}

Let $f\in L^1_{loc}([0,\infty);L^2(M))$ and define
\begin{equation}\label{eq:duhamel}
  \mathcal{I}(f)(t):=\int_{0}^t e^{i(t-s)\Delta} f(s) ds
\end{equation}
for $t \geq 0$ and $\mathcal{I}(f)(t)=0$ otherwise. We have the
following linear estimate for the Duhamel term.
\begin{prop}\label{prop:inhom_est}
  Let $s \geq 0$ and $T>0$. For $f \in L^1([0,T);H^s(M))$ we have
  $\mathcal{I}(f)\in X^s([0,T))$ and
  \begin{equation*}
    \|\mathcal{I}(f)\|_{X^s([0,T))}
    \leq \sup    \int_0^T \int_{M}f(t,x)\overline{v(t,x)} dxdt ,
  \end{equation*}
  where the supremum is taken over all $v \in Y^{-s}([0,T))$ with
  $\|v\|_{Y^{-s}}=1$.
\end{prop}
Details can be found in \cite{HTT10a}.  The following elementary
estimates will be used in the estimate of the nonlinear terms.
\begin{lemma}
  Let $\delta>0$. Then for every sequence
  $(c_\mu)$ and every $\lambda \geq  1$ one has
  \begin{equation}\label{dve}
    \Big|\sum_{\mu\lesssim\lambda}\Big(\frac{1}{\mu}+\frac{\mu}{\lambda}\Big)^{\delta}c_{\mu}\Big|^2
    \lesssim \sum_{\mu}|c_{\mu}|^2\,,
  \end{equation}
  where the sum runs over the dyadic values of $\mu$.
\end{lemma}
The proof of   \eqref{dve} follows by an application of the Cauchy-Schwarz
 inequality in $\mu$.

We now state the nonlinear estimate yielding Theorem~\ref{thm1}.
\begin{prop}\label{gi}
  Let $s\geq 1$ be fixed. Then, for $u_k\in X^s([0,1))$,
  $k=1,\ldots,5$, the estimate
  \begin{equation*}
    \Big\|\mathcal{I}(\prod_{k=1}^3
    \widetilde{u}_k)\Big\|_{X^s([0,1))}
    \ls \sum_{j=1}^3\|u_j\|_{X^s([0,1))}\prod_{\genfrac{}{}{0pt}{}{k=1}{k\not=j}}^3\|u_k\|_{X^1([0,1))},
  \end{equation*}
  holds true, where $\widetilde{u}_k$ denotes either $u_k$ or
  $\overline{u}_k$.
\end{prop}
\begin{proof} Set $I=[0,1)$.  Proposition \ref{prop:inhom_est} implies
  that we need to prove the multilinear estimate
  \begin{equation}\label{6lin}
    \Big| \int_{I \times M} \prod_{k=0}^4 \tilde u_k \ dx dt\Big|
    \lesssim  \| u_0\|_{Y^{-s}(I)}  \sum_{j=1}^3\Big(\|u_j\|_{X^s(I)}\prod_{\genfrac{}{}{0pt}{}{k=1}{k\not=j}}^3\|u_k\|_{X^1(I)}\Big)
  \end{equation}
  We dyadically decompose
  \[
  \widetilde{u}_k=\sum_{\lambda_k\geq 1} P_{\lambda_k} \widetilde{u}_k.
  \]
  In order for the integral in \eqref{6lin} to be nontrivial, the two
  highest frequencies must be comparable. By the
  Cauchy-Schwarz inequality and symmetry it suffices to show that
  \begin{equation}\label{6lina}
    \begin{split}
      S&= \sum_{ \Lambda }\|P_{\lambda_1}\widetilde{u}_1P_{\lambda_3}\widetilde{u}_3\|_{L^2}\|P_{\lambda_0}\widetilde{u}_0P_{\lambda_2}\widetilde{u}_2\|_{L^2}
      \\& \lesssim \|
      u_0\|_{Y^{-s}(I)}\|u_1\|_{X^s(I)}\|u_2\|_{X^1(I)}\|u_3\|_{X^1(I)},
    \end{split}
  \end{equation}
  where $\Lambda$ is as the set of all $4$-tuples
  $(\lambda_0,\lambda_1,\lambda_2,\lambda_3)$ of dyadic numbers $\geq 1$ satisfying
  \[
  \lambda_3\lesssim  \lambda_2 \lesssim \lambda_1 \quad \lambda_1 \sim \max \{\lambda_0,\lambda_2\}
  \]
  If $\lambda_1\sim \lambda_0$ then applying Proposition~\ref{thm:bil-str-ext}
  gives
  \begin{multline*}
    S\lesssim \sum_{ \Lambda} 
    \Big(\frac{\lambda_3}{\lambda_1}+\frac{1}{\lambda_3}\Big)^\delta
    \Big(\frac{\lambda_2}{\lambda_0}+\frac{1}{\lambda_2}\Big)^\delta
    \\
    \lambda_0^{-s}\| u_0\|_{Y^{0}(I)} \lambda_{1}^{s}\|u_1\|_{X^0(I)}
    \lambda_{2}\|u_2\|_{X^0(I)}\lambda_{3}\|u_3\|_{X^0(I)}\,.
  \end{multline*}
  Now we use \eqref{dve} to first sum in $\lambda_2$ and $\lambda_3$ and then
  Cauchy-Schwarz to sum with respect to $\lambda_0 \sim \lambda_1$ which yields
  \eqref{6lina}.

  If $\lambda_1 \sim \lambda_2$ then $\lambda_0 \lesssim \lambda_2$ therefore applying
  Proposition~\ref{thm:bil-str-ext} gives
  \begin{multline*}
    S\lesssim \sum_{ \Lambda} \Big( \frac{\lambda_0}{\lambda_1}\Big)^{s+1}
    \Big(\frac{\lambda_3}{\lambda_1}+\frac{1}{\lambda_3}\Big)^\delta
    \Big(\frac{\lambda_0}{\lambda_2}+\frac{1}{\lambda_0}\Big)^\delta
    \\
    \lambda_0^{-s}\| u_0\|_{Y^{0}(I)} \lambda_{1}^{s}\|u_1\|_{X^0(I)}
    \lambda_{2}\|u_2\|_{X^0(I)}\lambda_{3}\|u_3\|_{X^0(I)}\,.
  \end{multline*}
  Now we use \eqref{dve} to first sum in $\lambda_0$ and $\lambda_3$ and then
Cauchy-Schwarz   to sum with respect to $\lambda_2 \sim \lambda_1$ which again yields
\eqref{6lina}. This completes the proof of Proposition~\ref{gi}.
\end{proof}
With Proposition~\ref{prop:linear} and Proposition~\ref{gi} at our
disposal we can finish the proof of Theorem~\ref{thm1} exactly as in
\cite{HTT10a}.
\section{Proof of Theorem~\ref{thm2}}\label{sect:mult-str}
\noindent 
Before we turn to the proof of Theorem \ref{thm2} we discuss a lemma
which generalizes \cite[Lemma~2.1]{TT01}.
\begin{lemma}\label{lem:tt}
 Let $c \geq 0$, $d,e \in \R$. Then for all $k\geq 1$ we have
    \begin{equation}
      \label{eq:tt}\sup_{c,d,e \in \R} |\{(\xi,n)\in \R\times \Z: c\leq
      (\xi-d)^2+(n-e)^2\leq c+k\}| \ls k,
    \end{equation}
    where $|\cdot |$ denotes the product measure of the one
    dimensional Lebesgue and counting measure.
\end{lemma}
\begin{proof}
  The estimate \eqref{eq:tt} generalizes  \cite[Lemma
  2.1]{TT01}, which implies the claim if $d=e=0$ or $d=0,\ e=1/2$.
  By translation invariance it suffices to consider $d=0$ and $e\in
  [0,1)$. It suffices to prove the result when $k = 1$.
We may also assume that $c\geq 1$, because otherwise the
  set is contained in a ball of radius $2$. Following
  \cite{TT01}, let us define
  \[
h(x)=|\{(\xi,n)\in \R\times \Z: \xi^2+(n-e)^2\leq x\}|, \quad
  x\geq 1.
\] 
This is given by 
  \[
h(x)=2\sum_{n=-\lfloor \sqrt{x}-e \rfloor}^{\lfloor \sqrt{x}+e
    \rfloor}\sqrt{x-(n-e)^2}.
\] 
We need to estimate the quantity
  \begin{align*}
  V =  &|\{(\xi,n)\in \R\times \Z: c\leq
    \xi^2+(n-e)^2\leq c+1\}| = h(c+1)-h(c).
  \end{align*}
 We split it into two, $V = 2(S_1+S_2)$, where
  \[
  S_1=\sum_{n \in A_1 } \sqrt{c+1-(n-e)^2}, 
\] 
respectively 
  \[
  S_2=\sum_{n \in A_2} \sqrt{c+1-(n-e)^2} -\sqrt{c-(n-e)^2},
\]
with
\[
A_1 = \{ n \in \Z;\
  c \leq (n-e)^2 \leq c+1\}, \quad A_2 = \{ n \in \Z;\
   (n-e)^2 \leq c \}
\]
The sum $S_1$ has at most two terms, both less than $1$.
 For $S_2$ we have
  \begin{align*}
    S_2\ls &\ 2+ \int_{-\sqrt{c}+e}^{ \sqrt{c}+e}\sqrt{c+1-(t-e)^2} -\sqrt{c-(t-e)^2}dt \\
    \ls& \ 2+  \int_{-\sqrt{c}+e}^{
      \sqrt{c}+e}\frac{1}{\sqrt{c-(t-e)^2}}dt\ls 1,
  \end{align*}
and the proof is complete.
\end{proof}
Let us now turn to the proof of Theorem~\ref{thm2}.
 It suffices to prove the estimate
  \[
\|v\|_{L^4([0,1]\times M)}\ls
  \lambda^{\frac12-\delta}\mu^{\delta}\|\phi\|_{L^2(M)}, \quad
  v(t)=\psi(t)P_R e^{it\Delta}\phi,
\]
for $R \in \mathcal R(\mu,\lambda)$. 
We take $R$ as in \eqref{rlmu}.
By Galilean invariance we can translate $R$ to the origin, i.e. set $\xi_0=0$
and $c=0$. Thus we work with 
\begin{equation}\label{newr}
R = \{ \xi \in \widehat M: |\xi|^2 \leq \lambda, \ |a \cdot \xi| \leq \mu \}.
\end{equation}
 The above bound  follows from
  \[
\|v^2\|_{L^2(\R\times M)}\ls
  \lambda^{1-\delta'}\mu^{\delta'}\|\phi\|_{L^2(M)}^2,
  \quad \delta'=2\delta>0.
\] 
  The Fourier transform of $v$ is
  \[
  \mathcal{F}v(\tau,\xi)=\widehat{\psi}(\tau+|\xi|^2)\widehat{\phi}(\xi).
  \]
  Then the Fourier transform of $v^2$ satisfies
  \[
|  \mathcal{F}(v^2)(\tau,\xi)| \lesssim 
\int_{A(\tau,\xi)}
  |\widehat{\phi}(\sigma,\eta)||\widehat{\phi}(\tau-\sigma,\xi-\eta)|
  d\sigma d\eta
  \]
  where the integration in $\eta$ is performed with respect to the product
  measure of the Lebesgue measure on $\R^m$ and counting measure on
  $\Z^n$.  The set $A(\tau,\xi)$ is defined as the set of all
  $(\sigma,\eta)\in\R\times \widehat{M}$ such that
  \[
  |\sigma+|\eta|^2|\lesssim 1,\  |\tau-\sigma+|\xi-\eta|^2| \lesssim 1,\  \eta \in R,\ 
 \xi-\eta \in R
  \]
  for $\widehat{M}=\R^{m}\times \Z^{n}$, $m+n=4$, $n=1$ or $n=2$.
  Plancherel's theorem and the Cauchy-Schwarz inequality imply
  \[
  \|v^2\|_{L^2(\R\times M)}^2\ls
  \sup_{\tau,\xi}|A(\tau,\xi)|\|\phi \|^2_{L^2(\R\times
    M)}.
  \]
  so the claim is reduced to the estimate
  \[
\sup_{\tau,\xi}|A(\tau,\xi)|\ls 
  \lambda^{2-\delta''}\mu^{\delta''}, \quad \delta''=2\delta'>0.
\]
For each $\eta \in R$, $\sigma$ must be inside an interval of length
at most $2$. Thus we obtain 
  \[
  |A(\tau,\xi)|\ls |B(\tau,\xi)|
  \]
where the set $B(\tau,\xi)$ is defined as
\[
B(\tau,\xi)=\{\eta \in R: |\tau+|\sigma|^2+|\xi-\eta|^2|\ls
  1\}.
\]

{\it Case i)} $m=3$, $n=1$, i.e. $\widehat{M}=\R^3\times \Z$. For $R$ we 
use the relations \eqref{newr}.  Our problem is invariant with respect to
rigid rotations in $\R^3$.  Hence without any restriction in generality we 
can assume that the vector $a$ in  \eqref{newr} satisfies $a_2 = a_3 = 0$.
Then we can write
\begin{multline*}
B(\tau,\xi)=
\{\eta \in \widehat M:\ 
|a_1 \eta_1+a_4 \eta_4|\leq \mu, \
 |\eta|\leq \lambda, \
 |\tau+|\xi|^2+|\xi-\eta|^2|\ls  1\}.
\end{multline*}
Then, we observe that
\[
\sup_{\tau,\xi}|B(\tau,\xi)|\leq  |I_1| \sup_{\tau',\xi}|I_2(\tau',\xi)|
\]
where
\begin{multline*}
I_1=\{(\eta_1,\eta_4)\in \R\times \Z:\ |a_1 \eta_1+a_4\eta_4|\leq \mu ,\ 
|\eta_1|\leq \lambda, |\eta_4| \leq \lambda\},
\end{multline*}
and
\begin{multline*}
I_2(\tau',\xi)=\{(\eta_2,\eta_3)\in \R^2:\ 
|\tau'+\eta_2^2+\eta_3^2+(\xi_2-\eta_2)^2+(\xi_3-\eta_3)^2 |\ls 1 \}.
\end{multline*}
Let us first estimate $|I_1|$. Recall that $a_1^2+ a_4^2=1$.  If
$a_1^2\geq 1/2$ then the number of possible $\eta_4$ is $\lesssim
\lambda$ and then for fixed $\eta_4$ the Lebesgue measure of the
possible $\eta_1$ is $\lesssim \mu$.  If $a_4^2\geq 1/2$ we make
the same reasoning by replacing the role of $\eta_4$ and $\eta_1$.
In any case, $|I_1|\ls \mu\lambda$.  We also observe that the
constraint on $\eta_2$ and $\eta_3$ in $I_2$
is equivalent to
\[
\left|\tau'+\frac12(\xi_2^2+ \xi_3^2)+ 2\Big(\eta_2-\frac{\xi_2}{2}\Big)^2+
2\Big(\eta_3-\frac{\xi_3}{2}\Big)^2\right|\ls 1,
\]
which implies that $I_2(\tau',\xi)$ is contained in a circle or
annulus with area $\ls 1$ independently of $\tau'$ and $\xi$.
Summarizing the above discussion yields
\[
\sup_{\tau,\xi}|A(\tau,\xi)|\ls  \mu \lambda. 
\]
This concludes the proof of \eqref{strichartz} with $\delta=1/4$
(which corresponds to $\delta''=1$).

  {\it Case ii)} $m=n=2$, i.e. $\widehat{M}=\R^2\times \Z^2$.
In this case our problem is invariant with respect to rigid rotations
in $\R^2$. Hence without any restriction in generality we can assume
that $a_2=0$. Then we have
\begin{multline*}
B(\tau,\xi)=
\{\eta \in \R^2\times \Z^2:\ 
|a_1 \eta_1+a_3\eta_3+a_4\eta_4|\leq \mu,\
 |\eta|\leq \lambda, 
 \\
 |\tau+|\eta|^2+|\xi-\eta|^2| \ls 1\}.
\end{multline*} {\it Subcase a)} $a_1\geq
\frac12(\mu/\lambda)^{\frac13}$. The change of variables
$(\eta_1,\eta_2,\eta_3,\eta_4) \to (\zeta_1,\eta_2,\eta_3,\eta_4)$,
given by
\[
\zeta_1 = a_1 \eta_1+a_3\eta_3+a_4\eta_4 
\]
yields
\[ 
\sup_{\tau,\xi}|B(\tau,\xi)|
\leq a_1^{-1} |I_1| \sup_{\tau',\zeta_2,\zeta_3}|I_2(\tau',\zeta_2,\zeta_3)|
\]
where
\begin{align*}
I_1=&\ \{(\zeta_1,\eta_4)\in \R\times \Z:\ |\zeta_1|\leq \mu, |\eta_4|\leq \lambda\},\\
I_2=&\ \{(\eta_2,\eta_3)\in \R\times \Z:\  |\tau'+2(\eta_2-\zeta_2)^2+2(1+a_3^2/a_1^2)(\eta_3-\zeta_3 )^2|\ls 1 \}.
\end{align*}
First, we obviously have $|I_1|\ls \mu\lambda$. Secondly, by dilating
the $\eta_2$ variable by a factor of  $(1+a_3^2/a_1^2)^{\frac12}$
the relation defining $I_2$ takes the form
\[
  |\tau'+2(1+a_3^2/a_1^2)((\eta_2-\zeta_2)^2+(\eta_3-\zeta_3 )^2)|\ls 1 
\]
Dividing this by $2(1+a_3^2/a_1^2)$ we are in a position to 
 apply Lemma~\ref{lem:tt} with $k=1$. Scaling back  we obtain
\[
|I_2| \lesssim (1+a_3^2/a_1^2)^{\frac12}
\]
Thus 
\[
\sup_{\tau,\xi}|B(\tau,\xi)|\ls   a_1^{-1} (1+a_3^2/a_1^2)^{\frac12}
 \mu \lambda \lesssim \lambda^{\frac53} \mu^\frac13.
\]

{\it Subcase b)} $a_1 \leq \frac12 (\mu/\lambda)^{\frac13}$. 
Then we write
\[
\sup_{\tau,\xi}|B(\tau,\xi)|
\leq  |I_1| \sup_{\tau',\zeta_1,\zeta_2}|I_2(\tau',\zeta_1,\zeta_2)|
\]
where 
\[
\begin{split}
  I_1=& \ \{ (\eta_3,\eta_4) \in \Z^2:\ |\eta_3| +|\eta_4| \lesssim
  \lambda, \ |a_3 \eta_3 + a_4 \eta_4|\lesssim \mu^\frac13
  \lambda^\frac23.\} \\
 I_2=& \ \{ (\eta_1,\eta_2) \in \R^2: (\eta_1-\zeta_1)^2+(\eta_2-\zeta_2)^2 |\ls 1 \}.
\end{split}
\]
Since in this case we have $a_3^2+a_4^2 \approx 1$, it follows that
$I_1$ is contained in a rectangle of size $\lambda \times
\mu^{\frac13}\lambda^{\frac23}$. Hence $|I_1| \lesssim
\mu^{\frac13}\lambda^{\frac53}$. On the other hand $I_2$ is a circle
or annulus in $\R^2$ of area $\ls 1$. All in all, we conclude
\[
\sup_{\tau,\xi}|B(\tau,\xi)|\ls\mu^{\frac13}\lambda^{\frac53}.
\]
The claim \eqref{strichartz} follows with $\delta=\delta''/4=1/12$.
This completes the proof of Theorem~\ref{thm2}.

\bibliographystyle{amsplain} \label{sect:refs}\bibliography{nls-refs}

\end{document}